\documentclass[11pt]{article}
\usepackage[utf8]{inputenc}
\usepackage{amsmath}
\usepackage{amsfonts}
\usepackage{amssymb}
\usepackage{amsfonts}
\usepackage{amsthm}
\usepackage{amscd}
\usepackage{color}
\usepackage{graphics}
\usepackage{graphicx}
\usepackage[toc,page]{appendix}

\usepackage[all]{xy}

\usepackage{algorithm,algorithmic}

\usepackage{concmath}
\usepackage[T1]{fontenc}

\usepackage{tikz}

\newtheorem{theorem}{\bf Theorem}[section]
\newtheorem{lemma}[theorem]{\bf Lemma}
\newtheorem{proposition}[theorem]{\bf Proposition}

\newtheorem{remark}[theorem]{\bf Remark}

\newtheorem{example}[theorem]{\bf Example}
\newtheorem{corollary}[theorem]{\bf Corollary}

\newtheorem{problem}[theorem]{\bf Problem}

\newcommand{\cis}[1][C]{\ensuremath{\mathbb{#1}}}
\newcommand{\av}[1]{\ensuremath{\mathcal{#1}}}

\newcommand{\vek}[1][h]{\ensuremath{\mathbf{#1}}}
\newcommand{\f}[1]{\mathbf{#1}}

\newcommand{\pr}[1]{\ensuremath{\mathbb{P}^{#1}_{\cis}}}
\newcommand{\prR}[1]{\ensuremath{\mathbb{P}^{#1}_{\cis[R]}}}

\newcommand{\af}[1]{\ensuremath{\mathbb{A}^{#1}_{\cis}}}
\newcommand{\afR}[1]{\ensuremath{\mathbb{A}^{#1}_{\cis[R]}}}

\newcommand{\genus}{\ensuremath{\mathrm{g}}}

\newcommand{\rank}{\ensuremath{\mathrm{rank}}}
\newcommand{\gr}{\ensuremath{\mathrm{Gr}}}

\newcommand{\C}{\mathbb{C}}
\newcommand{\R}{\mathbb{R}}

\newcommand{\I}{\mathrm{i}}

\begin{document}
\title{Contour curves and isophotes on rational ruled surfaces}
\author{Jan Vr\v{s}ek\footnote{E--mail:\ vrsekjan@kma.zcu.cz}}
\date{}
\maketitle
\sloppy


\begin{abstract}
The ruled surfaces, i.e., surfaces generated by one parametric set of lines, are widely used in the~field of applied geometry. An~isophote on a surface is a curve consisting of surface points whose normals form a constant angle with some fixed vector. Choosing an angle equal to $\pi/2$ we obtain a special instance of a~isophote -- the so called contour curve. While contours on rational ruled surfaces are rational curves, this is no longer true for the isophotes. Hence we will provide a formula for their genus. Moreover we will show that the only surfaces with a~rational generic contour are just rational ruled surfaces and a one particular class of cubic surfaces. In addition we will deal with the reconstruction of ruled surfaces from their contours and silhouettes.\\

\end{abstract}

%
%

\section{Introduction}
Let $\av{X}$ be a surface in the projective space $\prR{3}$ and $\av{X}_{sm}$ denotes the set of its smooth points. Then for any fixed point $\f a\in\prR{3}$  the~\emph{contour} $\av{C}_\f a$ of $\av{X}$ with respect to a~viewpoint $\f a$ is defined as the closure of the set
\begin{equation}\label{eq contour}
  \{\f p\in\av{X}_{sm}:\ \f a\in T_\f p\av{X}_{sm} \},
\end{equation}
where $T_\f p\av{X}_{sm}$ denotes the~tangent plane at $\f p$. If $\av{H}$ is an~arbitrary plane not passing through $\f a$ then we may project a contour $\av{C}_\f a$ from the~point $\f a$ to the plane $\av{H}$. The projected curve is then the~so called \emph{silhouette} and usually denoted $\av{S}_\f a$, see Fig.~\ref{fig contour}. Some related studies on contours, silhouettes and their applications can be found e.g. in \cite{BiLa13,BiLaVr15,KiLe03b,SeKiKiEl06}

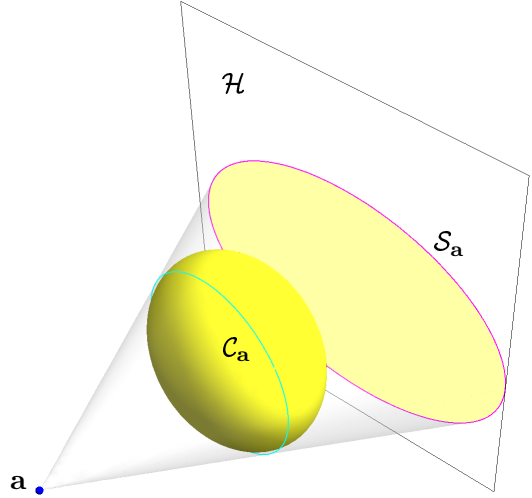
\begin{figure}[t]
\begin{center}
  \includegraphics[width=0.45\textwidth]{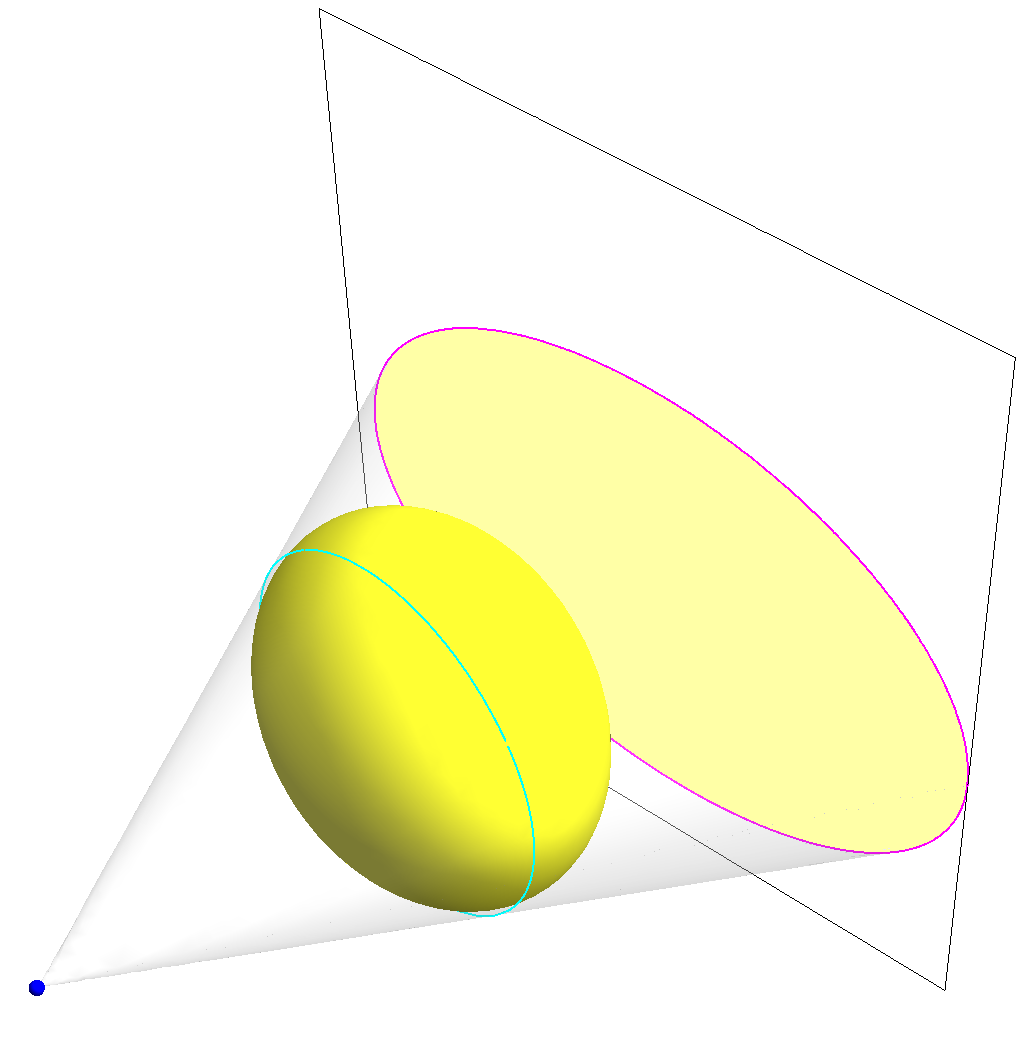}
  \begin{picture}(0,0)
\put(-200,10){$\f a$}
\put(-120,160){$\av{H}$}
\put(-40,100){$\av{S}_\f a$}
\put(-120,60){$\av{C}_\f a$}
\end{picture}
\hspace{4ex}

\begin{minipage}{0.9\textwidth}
\caption{The contour $\av{C}_\f a$ (cyan) w.r.t. the~point $\f a$ and the silhouette $\av{S}_\f a$ (magenta) as the projection of the contour into the plane $\av{H}$. \label{fig contour}}
\end{minipage}
\end{center}
\end{figure}

\begin{remark}\label{rem developable}
  If $\av{X}$ is a~developable surface, i.e. an envelope of one parametric set of planes, then the contour curves consist of unions of finite number of lines. In what follows we assume a surface to be non--developable.
\end{remark}

Let $(x_0:x_1:x_2:x_3)$ be coordinates in $\prR{3}$. Fix a hyperplane $\omega:x_0=0$ and the~absolute conic section $\Omega: x_0=x_1^2+x_2^2+x_3^2$ then the complement $\afR{3}=\prR{3}\backslash\omega$ is an~affine space endowed with the usual scalar product.  The~plane $\omega$ is called a plane at infinity and its points can be understood as directions in $\afR{3}$.  
We write $A=(a_1,a_2,a_3)$ for dehomogenization of a~point $(1:a_1:a_2:a_3)$ and $\overrightarrow{a}=(a_1,a_2,a_3)$ for dehomogenization of a~direction $\f a = (0:a_1:a_2:a_3)$. Depending on the position of the~point $\f a$, it is sometimes distinguished  between the~contour w.r.t a~central projection ($\f a\not\in\omega$) and a~parallel projection ($\f a\in\omega$). 

The~\emph{Gauss mapping} $\gamma:\av{X}\dashrightarrow(\prR{3})^\vee$,\footnote{The dashed arrow emphasizes the fact that the mapping $\gamma$ need not to be~defined for every point of the surface, but only on its dense subset} associated to a~surface $\av{X}\subset\prR{3}$,  assigns to a point of the surface its tangent plane $\gamma:\f p\mapsto T_\f p \av{X}$, viewed as a point in the~dual space $(\prR{3})^\vee$. If $\av{X}$ is given implicitly  by a homogeneous polynomial equation $F(x_0,x_1,x_2,x_3)=0$ then the formula for the~Gauss mapping is just
\begin{equation}\label{eq gauss mapping}
  \gamma:\f p\mapsto\left(\partial_{x_0}F(\f p):\cdots:\partial_{x_3}F(\f p)\right).
\end{equation}
The image of a~non-developable surface under the~Gauss map is again an~algebraic surface -- the so called \emph{dual surface}  $\av{X}^\vee$.
The \emph{normal mapping} of the surface $\nu:\av{X}\dashrightarrow\prR{2}$ assigns to a point of the surface a normal direction at this point. To define it in a projective settings just assign to a point $\f p\in\av{X}$ the point polar to the line $T_\f p\av{X}\cap \omega$ with respect to a polarity induced by the~absolute conic section $\Omega$. With the choice of $\Omega$ we have made, the normal mapping is 
\begin{equation}\label{eq normal mapping}
  \nu:\f p\mapsto\left(\partial_{x_1}F(\f p):\cdots:\partial_{x_3}F(\f p)\right),
\end{equation}
i.e., it can be viewed as the~composition of the~Gauss mapping and the~projection from the point $(1:0:0:0)$. It is easily verified that this definition agrees with the~usual construction of the~normal vector on the affine patch $x_0\not=0$.

Unlike the contour the definition of an~isophote depends on the metric of the~ambient space. It is defined as a loci of points where the surface normals encloses a~constant angle with a~fixed vector.  This definition, usual in differential geometry, c.f. \cite{DoYa15,KiLe03,Po84}, is not suitable when attacking the problem with the~algebraic techniques. The reason is that  the isophote would not be an algebraic curve, in this case, but only its half. Hence we modify the definition slightly. The  \emph{isophote} $\av{I}_{\f a,\alpha}$ is the closure of the set
\begin{equation}\label{eq isophote}
  \left\{\f p\in\av{X}_{sm}: (\nu(\f p)\cdot\f a)^2-\alpha^2(\f a\cdot\f a)(\nu(\f p)\cdot\nu(\f p))=0\right\},
\end{equation}
where $\f a=(a_1,a_2,a_3)\in\prR{2}$, is a~given direction, $\phi=\arccos\alpha$ is the angle and $\f x\cdot\f y=\sum_i x_iy_i$.
Hence $\av{I}_{\f a,\alpha}$ is a set of points on the surfaces where the~normal direction forms  angles $\pm\phi$ with the~direction $\f a$. Considering $\f a$ to be a point in $\omega$, i.e. $\f a=(0:a_1:a_2:a_3)$ and  letting $\alpha=\cos\pi/2=0$ we obtain exactly a contour curve $\av{C}_\f a=\av{I}_{\f a,0}$.

In geometric modelling, curves and surfaces are usually given by their polynomial or rational parameterization. Hence they are special instances of a wide class of algebraic varieties. Since contours and isophotes on algebraic surfaces are algebraic curves as well it makes sense to study them via techniques from algebraic geometry.  Because of the used methods we will replace the field $\R$ in definitions by $\C$. It is clear that \eqref{eq contour} and \eqref{eq isophote} still make a perfect sense. However in engineering applications one is more interested in results about real surfaces -- algebraic surfaces with real dimension two defined be real equations. Hence we will discuss the consequences of our results for real varieties as well.   

Since the paper aims to the contour curves and isophotes on rational ruled surfaces we will recall some basic facts about ruled surfaces,  for more details see e.g. \cite{PoWa01}. A~\emph{rational ruled surface} $\av{R}$ is a surface in $\pr{3}$  generated by a~rational one-parametric family of lines -- the so called \emph{rulings}. Hence it admits a~parameterization
\begin{equation}\label{eq R}
\vek[r](s,t_0,t_1)=t_0\vek[p](s)+t_1\vek[q](s)=(t_0p_0(s)+t_1q_0(s):\cdots:t_0p_3(s)+t_1 q_3(s)), 
\end{equation}
where $p_i(s)$ and $q_i(s)$ are polynomials. The rulings on the surface are parametric curves corresponding to a fixed parameter$s$. The rational curves $\f p(s)=\f x(s,1,0)$ and $\f q(s)=\f x(s,0,1)$ intersect a generic ruling exactly once. A curve on $\av{R}$ with this property is called a~\emph{section} and it can be seen that each section is a rational curve. In fact for every rational ruled surface $\av{R}$ in $\pr{3}$ there exist  numbers $m,n$ such that $\av{R}$ is a projection of  ruled surface $\hat{\av{R}}\subset\pr{2m+n+1}$ parameterized as 
\begin{equation}
\left(t_0:t_0s:\dots:t_0s^m:t_1:t_1 s:\dots t_1 s^{m+n}\right)
\end{equation}
The projection is birational and preserves the degrees of curves whose images are not contained in the singular locus of $\av{R}$. Two sections on $\hat{\av{R}}$ of degrees $m$ and $m+n$ have the minimal possible  degrees   and the degree of the surface $\av{R}$ is then $2m+n$. 

A~rational parametrization of the surface $\av{R}$ can be obtained by joining the~corresponding points on arbitrary two sections by the~line. 
\begin{proposition}\label{prp coherent parameterizations}
  Let $\f p(u)$ and $\f q(v)$ be proper parameterizations of two sections on ruled surfaces. Then there exist reparameterizations $\phi(s)$ and $\psi(s)$, such that $t_0\f p(\phi(s))+t_1\f q(\psi(s))$ parametrizes the surface. Moreover $\phi$ and $\psi$ can be chosen to be linear fractional transformations, i.e., in the rational functions $\frac{\alpha s+\beta}{\gamma s+\delta}$.
\end{proposition}


A~section of a~special interest on $R$ is a~section by a plane. If we choose $\f q(s)=(0:q_1(s):q_2(s):q_3(s))$ to be a parameterization of the~section by $\omega$, then \eqref{eq R} provides a~parameterization of the~affine piece of the surface in the usual form
\begin{equation}\label{eq R affine}
  R(s,t)=P(s)+t\overrightarrow{q}(s).
\end{equation}

Denote $\f x=(x_0:x_1:x_2:x_3)$ and let $\dot{\f p}$, $\dot{\f q}$ be the derivatives of $\f p$ and $\f q$, respectively. Then the tangent plane at the~point $\f r(s,t_0,t_1)$ is spanned by $\f p(s)$, $\f q(s)$ and $t_0\dot{\f p}(s)+t_1\dot{\f q}(s)$, i.e. it possesses an equation 
\begin{equation}\label{eq tangent planes}
  t_0\det[\f x,\f p(s),\f q(s),\dot{\f p}(s)]+t_1 \det[\f x,\f p(s),\f q(s),\dot{\f q}(s)]=0.
\end{equation}
Assume $\f p(s_0)\not=\f q(s_0)$ then the ruling corresponding to $s=s_0$ is called \emph{regular}, \emph{torsal} or \emph{singular}  if $\rank[\f p(s_0),\dot{\f p}(s_0),\f q(s_0),\dot{\f q}(s_0)]$ equals to  $4$, $3$, or $2$, respectively -- see \cite{PoWa01} for the~detailed description. On a~non-developable surface there is at most finitely many non-regular rulings.  Since \eqref{eq tangent planes} is linear in $t_0$ and $t_1$ the tangent planes along a regular ruling form a line in the dual space. It turns out that dual surface of a~non-developable ruled surface is a~ruled surface as well. Moreover it is known that $\deg\av{R}^\vee=\deg(\av{R})$.

\section{Contour curves}
\subsection{Contour curves in general}

As already mentioned, curves and surfaces in geometric modelling are usually given by their rational parameterizations. Hence for a given surface one would like to have a formula for a rational parameterization of its contours. Conversely in \cite{BiLa13} the rational contour curves were used to produce  rational parametrizations of canal surfaces. So the first question to ask is how many rational surfaces possesses rational contours. Let us start with a simple example. 

Quadratic patches are one of the simplest classes of rational parametric surfaces used in geometric modelling. These are the~projections of Veronese surface in $\pr{5}$ to $\pr{3}$. Depending on the projection (or equivalently on the number of base points of the~resulting parameterization) the quadratic patch parametrizes one of the following surfaces:
\begin{enumerate}
  \item quadric,
  \item ruled cubic with double line,
  \item Steiner surface  (of degree 4).
\end{enumerate}
As we will see, the contour curves on regular quadrics are rational and the same is true for ruled cubic surfaces. However a generic projection of the~Veronese surface and thus almost all quadratically parametrized surfaces in $\pr{3}$ are Steiner quartics. Their generic\footnote{Generic here means that there exists a Zariski open set $U\subset\pr{3}$ such that for all $\f a\in U$ the contour has the~genus $1$. Hence there still can exist rational contour curves on the surface, but their set has the~codimension at least one in $\pr{3}$.} contour curves are elliptic curves  and thus they  are not rational. Hence even a~very simple surfaces do not posses  contours parametrizable by the~standard techniques used in CAGD. It is already known that contours on rational ruled surfaces are rational, see e.g. \cite{PePoRa99}. The following theorem completes the lists of all such surfaces.

\begin{theorem}\label{thm surf with rational }
  A generic contour curve on a surface in $\pr{3}$ is rational if and only if the surface is rational ruled or the~Cayley cubic, i.e., rational cubic surface with four double points.
\end{theorem}
\begin{proof}
For a given $\f a\in\pr{3}$ a point $\f p\in\av{X}_{sm}$ is contained in the contour $\av{C}_\f a$ if and only if $\gamma(\f p)=(y_0:y_1:y_2:y_3)$ fulfils 
\begin{equation}
  y_0 a_0+y_1 a_1+y_2a_2+y_3a_3=0,
\end{equation}
i.e., contour curves are mapped to plane sections of $\av{X}^\vee$ via the~Gauss map. Since $\gamma$ is birational and $(\av{X}^\vee)^\vee=\av{X}$ by the~reflexivity theorem (see e.g. \cite[p. 208]{Ha92}), the surfaces with rational contour curves are exactly the duals of surfaces with rational plane sections. Such a surface is well known to be a~projection  of a~rational ruled surface or the~Steiner surface, see \cite{GrHa78}. As already mentioned, the dual of rational (non-developable) ruled surface is again a~rational ruled surface. The dual of the~Steiner surface is the~Cayley cubic -- the cubic surface with four ordinary double points see \cite[p. 449]{Do12}.
\end{proof}

If $F(x_0,\dots,x_3)$ is the defining polynomial of a~surface $\av{X}$ then the first polar with respect to point $\f a\in\pr{3}$ is the surface with the equation $\sum_{i=0}^3 a_i \partial F/\partial x_i=0$. The contour curve $\av{C}_\f a$ is nothing but the closure of the intersection of $\av{X}_{sm}$ with the polar surface. If $V(I)$ denotes the variety associated to an~ideal $I$ then there is the~well known identity $\overline{V(I)\backslash V(J)}=V(\sqrt{I}:J)$. Because the ideal of the~singular locus is generated by all the partial derivatives of $F(\f x)$ we arrive at the~ideal of  $\av{C}_\f a$
\begin{equation}\label{eq impl eq of contour}
 \sqrt{\left\langle F,\sum a_i \frac{\partial F}{\partial x_i}\right\rangle} : \left\langle \frac{\partial F}{\partial x_0},
 \frac{\partial F}{\partial x_1},
 \frac{\partial F}{\partial x_2},
 \frac{\partial F}{\partial x_3}\right\rangle.
\end{equation}

Conversely, let $\av{C}\subset\pr{3}$ be a curve defined as a~simultaneous solution of homogeneous equations $G_i(x_0,x_1,x_2,x_3)=0$ for $i=1,\dots,m$ and let $\f a\in\pr{3}$ be a fixed point. A surface $\av{X}$ (not necessarily rational and ruled at this moment)  containing the~curve $\av{C}$ as a~contour curve w.r.t. a~point $\f a$ possesses a defining polynomial $F=\sum_{i=1}^mH_iG_i$ for some homogeneous polynomials $H_i$. The following relation ensures that $\f a\in T_\f p\av{X}$
\begin{equation}\label{eq implicit contour}
  \sum_{j=0}^3\frac{\partial F(\f p)}{\partial x_j}a_j=\sum_{j=0}^3\sum_{i=1}^m H_i(\f p)\frac{\,\partial G_i(\f p)}{\partial x_j}a_j=0
\end{equation}
for each $\f p\in\av{C}$. Hence the expression defined by \eqref{eq implicit contour} must be contained in the ideal generated by $G_i$ and thus there exist homogeneous polynomials $L_i$ for $i=1,\dots, m$ such that

\begin{equation}\label{eq syzygy condition}
  \sum_{i=1}^m H_i\left(\sum_{j=0}^3\frac{\partial G_i}{\partial x_j}a_j\right)-\sum_{i=1}^m L_iG_i\equiv 0,
\end{equation}
The polynomials $(H_0,\dots,H_3,L_0,\dots,L_3)$ fulfilling \eqref{eq syzygy condition} forms the~so called syzygy module, see e.g. \cite{CLO05}.  Nevertheless restricting the attention to the surfaces of a~fixed degree reduces the problem to the system of linear equations. We illustrate this on the~example. 

\begin{example}\label{ex cubic}
Let $\av{C}: x_0^2-x_1^2-x_2^2+x_3^2=x_1^2+x_2^2+x_3^2-2x_0x_1=0$ be the complete intersection of two quadrics, i.e., an elliptic curve of degree 4. And let $\f a=(1:-1:0:-1)$. The goal is to find all cubic surfaces with $\av{C}$ as a contour curve w.r.t. $\f a$. The defining equation of $\av{X}$ can be written as $(x_0^2-x_1^2-x_2^2+x_3^2)H_1+(x_1^2+x_2^2+x_3^2-2x_0x_1)H_2$, where $H_1$ and $H_2$ are linear forms.  Set $H_1=\sum_{i=0}^3\alpha_ix_i$ and $H_2=\sum_{i=0}^3\beta_ix_i$. The polynomials $L_1$ and $L_2$ then must have degree 0, i.e. $L_1=\gamma$ and $L_2=\delta$. substituting to \eqref{eq syzygy condition} leads to
\begin{equation}
\begin{array}{l}
(-2\alpha_0+2\beta_0-\gamma)x_0^2+
( -2\alpha_1 -   4\beta_1 + \gamma + \delta)x_1^2+
(\gamma + \delta)x_2^2+
 (2 \alpha_3 - 2 \beta_3 - \gamma + \delta)x_3^2+  \\
  ( -2 \alpha_0 -  2 \alpha_1 - 4 \beta_0 + 2 \beta_1 - 2 \delta)x_0x_1+
 ( -2 \alpha_2 +  2 \beta_2) x_0x_2+
( 2 \alpha_0 - 2 \alpha_3 - 2 \beta_0 + 2 \beta_3)x_0x_3+\\
( -2 \alpha_2 -  4 \beta_2)x_1x_2+
 (2 \alpha_1 - 2 \alpha_3 - 2 \beta_1 - 4 \beta_3) x_1x_3+
( 2 \alpha_2 - 2 \beta_2)x_2x_3\equiv0.
\end{array}
\end{equation}
This system of 10 linear equations in variables $\alpha_i,\beta_i,\gamma,\delta$ has the~unique solution (up tu a~scalar multiple)  $H_1=x_0-2x_1-x_3$ and $H_2=x_0+x_1-x_3$. Thus there exists a unique cubic surface with a~given contour -- namely the surface 
\begin{equation}
  F=H_1G_1+H_2G_2=-x_0^3 - x_1^3 + 2 x_0 x_2^2 - x_1 x_2^2 + x_0^2 x_3 + 2 x_0 x_1 x_3 - 2 x_1^2 x_3 - 2 x_2^2 x_3 + 
 3 x_0 x_3^2=0.
\end{equation} 
Let us emphasize that the curve $\av{C}$ may be only a~component of the whole contour curve $\av{C}_\f a\subset\av{X}$.  In our case, the surface $\av{X}$ is non-singular and thus by \eqref{eq impl eq of contour} the contour with respect to $\f a$ is a common solution of
\begin{equation}
  F(x_0,x_1,x_2,x_3)=\sum_{i=0}^3\frac{\partial F(x_0,x_1,x_2,x_3)}{\partial x_i}a_i=0,
\end{equation} 
where $\f a=(a_0:\dots:a_3)$. B\'{e}zout theorem tells us that the degree of $\av{C}_\f a$ is six. Indeed in our particular case $\av{C}_\f a$ is the~union of $\av{C}$ and two lines $(s_0: s_1:\pm\sqrt{7/5} (s_0 + s_1): s_0 + 2 s_1)$, see Fig.~\ref{fig surf with given cont}.

\begin{figure}[t]
\begin{center}

  \includegraphics[width=0.5\textwidth]{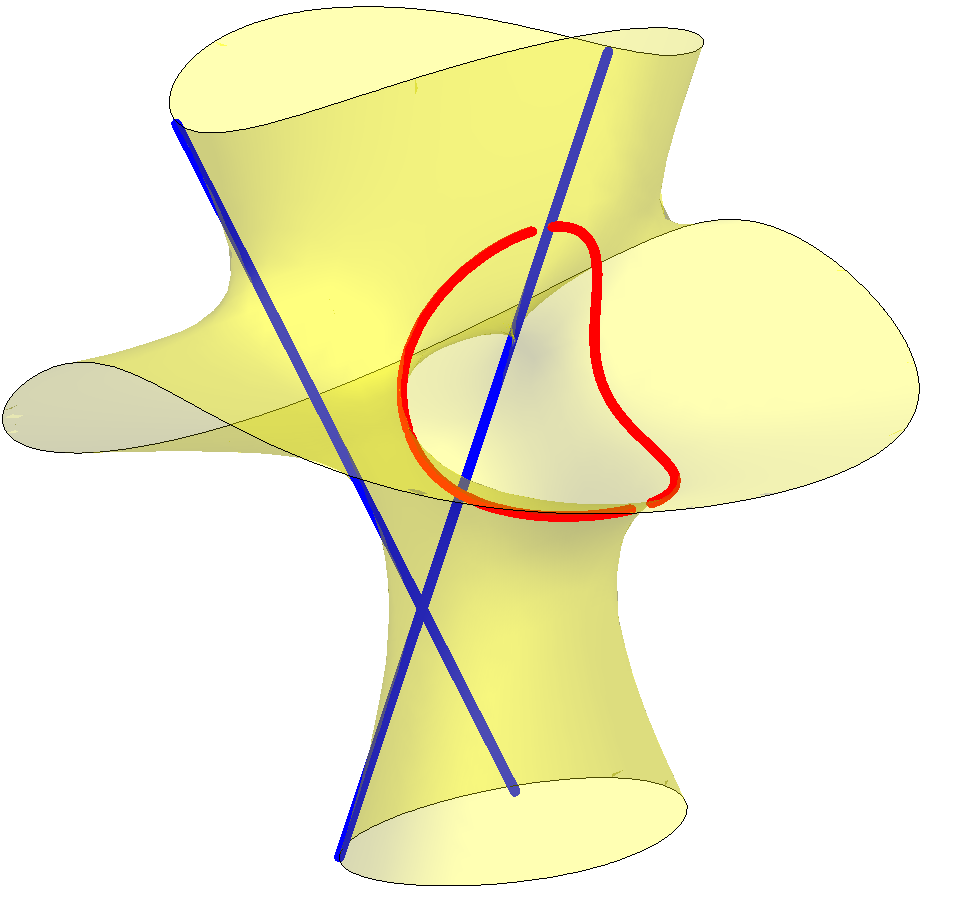}
\begin{minipage}{0.9\textwidth}
\caption{Cubic surface whose contour consists of curve $\av{C}: G_1=G_2=0$ (red) and two lines (blue). \label{fig surf with given cont}}
\end{minipage}
\end{center}
\end{figure}

\end{example}

The above approach can be used to find surfaces containing more than one contour lines.  Hence it is perfectly reasonable to ask how many contours determine a surface and how to reconstruct it from the given set. Clearly the answer will generally depend on the degree of the sought surface. Note that the introduced method of a~surface reconstruction has some drawbacks. As we have already seen, the given curve may be only a part of the~contour on the resulting surface. Second, let $F=0$ be the~equation of the above cubic surface and $G_1,G_2$ be two quadrics intersecting at $\av{C}$ from the Example~\ref{ex cubic}. If we solve \eqref{eq syzygy condition} for a~quartic surfaces then the set of solutions will contain polynomials $F\cdot H$, where $H$ is an~arbitrary linear form or the product $G_1\cdot G_2$. In fact the system of equations cannot be used to distinguish between regular solutions and the surfaces which are reducible or contain $\av{C}$ as a singular curve.

As usually the better knowledge of the geometry of sought surface can simplify the problem significantly. Hence one of our goals is  to provide the answer for the class of ruled surfaces.

\begin{problem}\label{prob contours}
  Determine  numbers $c(k)$ such that a rational ruled surface of degree $k$ is determined by $c(k)$ contour curves.
\end{problem}

The analogous problem is obtained by replacing the~contour curves by the~silhouettes. Recall (see Fig.~\ref{fig contour}) that the silhouette  can be understood as the~``boundary''  of the surface projected from the center $\f a$ to some chosen plane. Hence the~following problem is motivated by a~reconstruction of a~surface from its two--dimensional images.

\begin{problem}\label{prob outlines}
   Determine  numbers $s(k)$ such that a rational ruled surface of degree $k$ is determined by $s(k)$ silhouettes.
\end{problem}


\subsection{Contour curves on quadrics}
 
Although all quadrics are ruled surfaces we will treat them separately. There are two reasons for this. First, we can solve Problems \ref{prob contours} and \ref{prob outlines} directly without any~reference to the rulings. And second, the~quadrics are a~typical illustration of the~drawback of an~approach via complex numbers. Indeed the real part of the~sphere, paraboloid, etc. contains no line.  So one cannot consider them to be ruled surfaces from the point of view of real geometry. Fortunately we may prove

\begin{proposition}
  Real rational ruled surface contains one--parametric set of real lines or it is a quadric.
\end{proposition}

This result sounds so classically that it must be known already. However we did not find it in the~literature and thus we present its~proof for the sake of completeness.
 
\begin{proof}
Since $\av{X}$ is ruled it contains a family of complex lines. For  each such a line,  $\av{X}$ contains the complex conjugate line as well. Since the lines generating a surface moves in a complex family, its real dimension  is two and thus a generic real point is an intersection of two conjugate lines. Now we prove that through each (complex) point of the surface pass at least two lines. Let
\begin{equation}
  \Gamma:=\{(\f p,\av{L})\in\av{X}\times\gr(1,3)\ \mid \f p\in \av{L}\subset\av{X}\}.
\end{equation}
Denote by $\pi:\Gamma\rightarrow\av{X}$ the projection onto the first factor. The degree of the projection $\pi$  measures how many lines pass through a generic point of the surface. Since we know that for $\f p\in\av{X}(\R)$ the cardinality of the fiber is at least two and $\av{X}(\R)$ is not contained in any algebraic subset of dimension one on $\av{X}$ we conclude that there must exist a Zariski open subset of $\av{X}$ with the same cardinality. And thus $\deg\pi\geq 2$.

Now it remains to prove that a surface with at least two lines through each point is a plane or a quadric. If $\av{X}$ contains the one-dimensional family of lines and no other line then through a non-singular point there passes only one line. so there must exist at least two lines $\av{P,Q}\subset\av{X}$ not belonging to the family and intersecting the members of the family in one point (this follows from the intersection product on the linear normalization of $\av{X}$). Hence we have two sections of degree one on $\av{X}$  and the degree of the ruled surface is two or one depending whether lines $\av{P}$ and $\av{Q}$ are skew or intersecting.

Since real plane always contains real lines we conclude that the only possible real ruled surfaces without real family of lines are quadrics.
\end{proof}


According to Remark~\ref{rem developable} we will consider regular quadrics only. The~contour w.r.t. a~point $\f a$ is then the~intersection of the quadric with its polar plane w.r.t. the~point $\f a$.  It is a~regular conic section whenever $\f a$ is not a point on the quadric, or it consists of two lines otherwise. This confirms that a generic contour curve of a regular quadric is a rational curve.

Conversely let a  conic section $\av{C}$ be given as the intersection of a quadric $G(\f x)=0$ and a plane $H(\f x)=0$. The cone joining the~point $\f a$ and the~conic section $\av{C}$ (We will use notation $\f a\#\av{C}$ for such a cone.) possesses a~quadratic equation $F(\f x)=0$ obtainable from the system 
from \begin{equation}
  \f x=s \f a+t \f y, \quad G(\f y)=0\quad\text{and}\quad H(\f y)=0,
\end{equation}
after elimination of variables $\f y=(y_0:\dots: y_3)$, $s$ and $t$. 

\begin{lemma}\label{lem eqn of quadric}
  With the~above notation, the equation of any~quadric containing a conic section $\av{C}$ as a contour w.r.t. a point $\f a$ can be written as
\begin{equation}\label{eq quadric with given contour}
  \lambda F(\f x)+\mu H^2(\f x)=0,
\end{equation}
where $\lambda,\mu\in\C$ are free parameters.
\end{lemma}

\begin{proof}
The ideal of $\av{C}$ is generated by $F$ and $H$. Denote the equation of the~sought quadric $\av{Q}$ by $G(\f x)=0$. Since $\av{C}$ is contained in $\av{Q}$ it is possible to write $G(\f x)=\lambda F(\f x)+ L(\f x)H(\f x)$, where $\lambda\in\C$ and $L(\f x)$ is a~linear form. Moreover the cone $\f a\#\av{C}$ shares the tangent plane with $\av{Q}$ at every point of $\av{C}$. It follows that the intersection of  $\av{Q}$ with $\f a\#\av{C}$ is the conic section $\av{C}$ counted twice. At the same moment this intersection is nothing but section of the cone by two planes $H(\f x)=0$ and $L(\f x)=0$. Hence it must be $L(\f x)=\mu H(\f x)$.
\end{proof}

Since the equation of a~quadric has 10 coefficients and it is unique up to multiplication by a constant, the set of all quadrics can be identified with $\pr{9}$ in a usual way.
As shown above, the set of the~quadrics with a~given contour forms a~line in this parameter space. Recall that the set of singular quadrics is  a hypersurface of degree 4 in $\pr{9}$.  Hence there are four singular quadrics in each  pencil (not contained in the hypersurface). As usual the proper intersection multiplicities must be taken into account. For example reducible quadrics (union of two planes) form a singular locus of this hypersurface. They are all singular points of multiplicity two except of double planes $H^2(\f x)=0$ which are triple points. Thus double planes count as at least three singular quadrics in the pencil. From this we can see that the only singular quadrics in the pencil \eqref{eq quadric with given contour} are the cone and the double plane.

\begin{theorem}\label{thm quadric two contours}
 Two contour curves determine a quadric uniquely ($c(2)=2$).  
\end{theorem}
\begin{proof}
 Two lines corresponding to the two contour curves on the same quadric $\av{Q}$ cannot be skew because they both contain a point corresponding to $\av{Q}$. Hence it is enough to prove that they are not identical. Assume that they are the same, then by Lemma~\ref{lem eqn of quadric} we may parametrize the pencils in two ways -- say $\lambda_1 F_1+\mu_1 H_1^2$ and $\lambda_2 F_2+\mu_2 H_2^2$. Since  $H_1=0$ and $H_2=0$ are obviously different planes (otherwise the contours would be the same), we see that they intersect the hypersurface of singular quadrics in at least $3+3+1$ points. This is impossible unless each quadric in the pencil is singular which is a contradiction to the general assumption on regularity of the quadric $\av{Q}$.
\end{proof}


Let be given a conic section $\av{S}_\f a$ and a point $\f a\in\pr{3}$ not contained in the plane of $\av{S}_\f a$. Similarly  to Lemma~\ref{lem eqn of quadric}, it is possible to describe all the quadrics with $\av{S}_\f a$ as silhouettes w.r.t. $\f a$.
In particular, let $F_\f a(\f x)=0$  and $H_\f a(\f x)=0$ be the~defining equations of the cone $\f a\#\av{S}_\f a$ and an arbitrary plane not passing through $\f a$, respectively.  The intersection of the cone with the plane is a regular conic section $\av{C}_\f a$ whose projection from the center $\f a$  is the silhouette $\av{S}_\f a$. Hence $\av{C}_\f a$ can be a contour curve on the sought quadric, we immediately obtain its defining polynomial
\begin{equation}
  Q(\f x)=\alpha F_\f a(\f x)+H_\f a^2(\f x)
\end{equation}
for some nonzero constant $\alpha$. First, we observe that two silhouettes are not enough to reconstruct a quadric. Let $\av{S}_\f b$ be the~second silhouette and use analogous notation as above.  We would like to lift up both silhouettes to obtain contours $\av{C}_\f a$ and $\av{C}_\f b$ which determine the quadric.  Since the contours are conic sections on the common quadric they must intersect in two points $\f p$ and $\f q$. Moreover the~tangent plane to the quadric and thus the~tangent planes to the cones $\f a\#\av{S}_\f a$ and $\f b\#\av{S}_\f b$ at these points must contain both points $\f a$ and $\f b$. This allows to determine $\f p, \f q\in (\f a\#\av{S}_\f a)\cap(\f b\#\av{S}_\f b)$. Hence both planes $H_\f a(\f x)=0$ and $H_\f b(\f x)=0$ belong to the pencil of planes passing through the~line joining $\f a$ and $\f b$.

Assume that we have already found some quadric $Q(\f x)=0$ with the~given silhouettes. Hence we may write
\begin{equation}\label{eq Q outlines}
  Q(\f x)=\alpha F_\f a(\f x)+H_\f a^2(\f x)=\beta F_\f b(\f x)+H_\f b^2(\f x).
\end{equation}
Hence $\alpha F_\f a-\beta F_\f b=(H_\f b+H_\f a)(H_\f b-H_\f a)$ is a singular quadric in the pencil $\lambda F_\f a+\mu F_\f b$.  Moreover it is reducible and thus the pencil contains  singular quadrics for parameter values $(\lambda:\mu)$ equal to $(1:0)$, $(0:1)$ with multiplicity one and  $(\alpha:-\beta)$ with multiplicity two.

From this we finally derive the~equation $P(\f x)=0$ of an arbitrary quadric with the silhouettes $\av{S}_\f a$ and $\av{S}_\f b$.  Let us write it  again as
\begin{equation}\label{eq P outlines}
  P(\f x)=\alpha F_\f a(\f x)+G_\f a^2(\f x)=\beta F_\f b(\f x)+G_\f b^2(\f x),
\end{equation}
where $\alpha$, $\beta$ are same as in~\eqref{eq Q outlines} and $G$'s are linear forms belonging to the same pencils as $H$'s. Hence we may write $G_\f a=\alpha_1 H_\f a+\alpha_2 H_\f b$ and $G_\f b=\beta_1H_\f a+\beta_2H_\f b$. Substituting this into \eqref{eq P outlines} and subtracting from \eqref{eq Q outlines} we arrive at
\begin{equation}
 \beta_1=\sqrt{\alpha_1^2-1},\quad\beta_2=\sqrt{\alpha_1^2+1}\quad\text{and}\quad \alpha_2=\frac{\alpha_1^4-1}{\alpha_1}.
\end{equation}
since there exists a one-parameter family of quadrics with silhouettes $\av{S}_\f a$ and $\av{S}_\f b$ -- see Fig.~\ref{fig outlines of quadric}, at least three silhouettes are needed. 

\begin{figure}[t]
\begin{center}

  \includegraphics[width=0.34\textwidth]{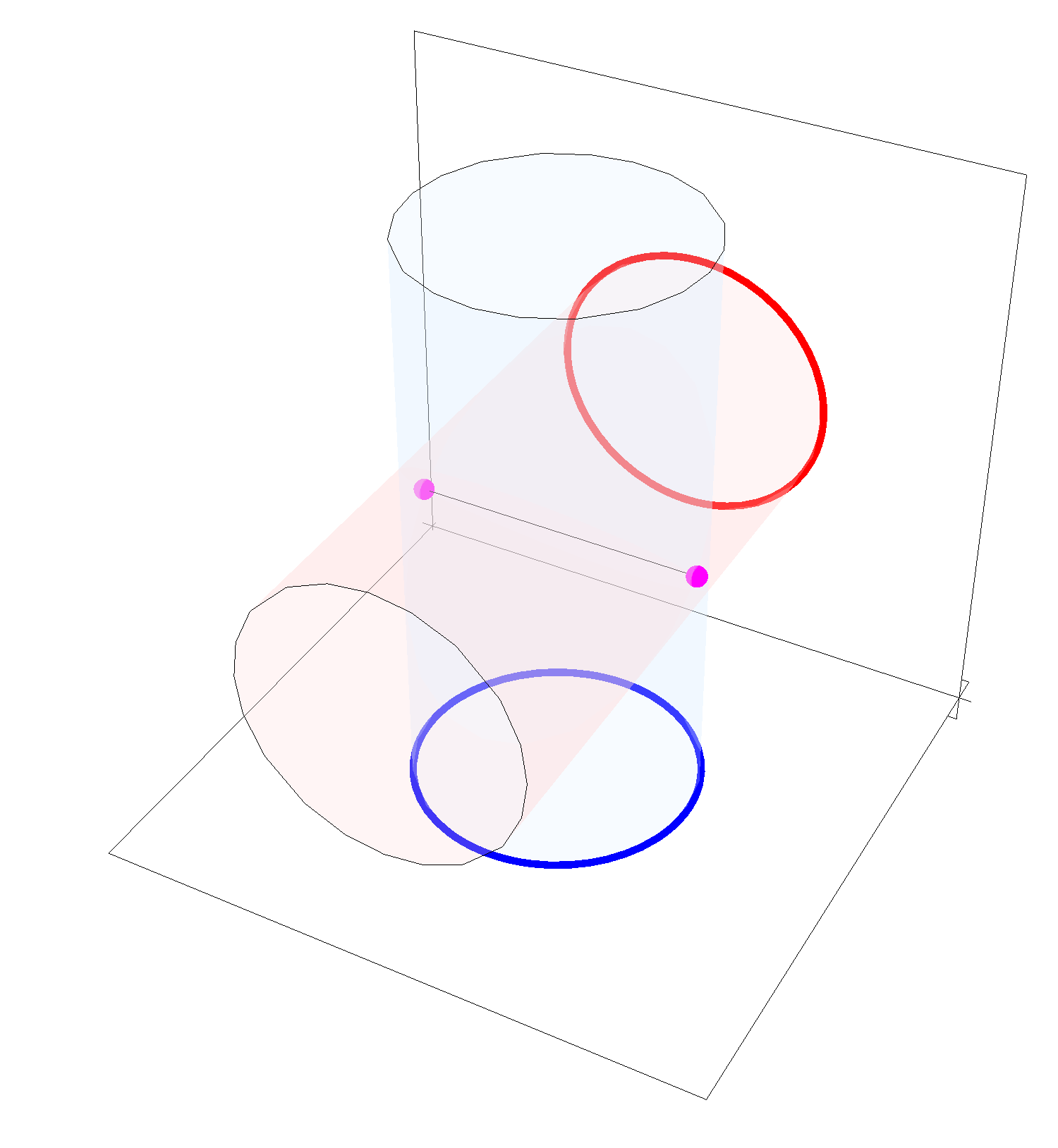}\hspace{3ex}
  \includegraphics[width=0.34\textwidth]{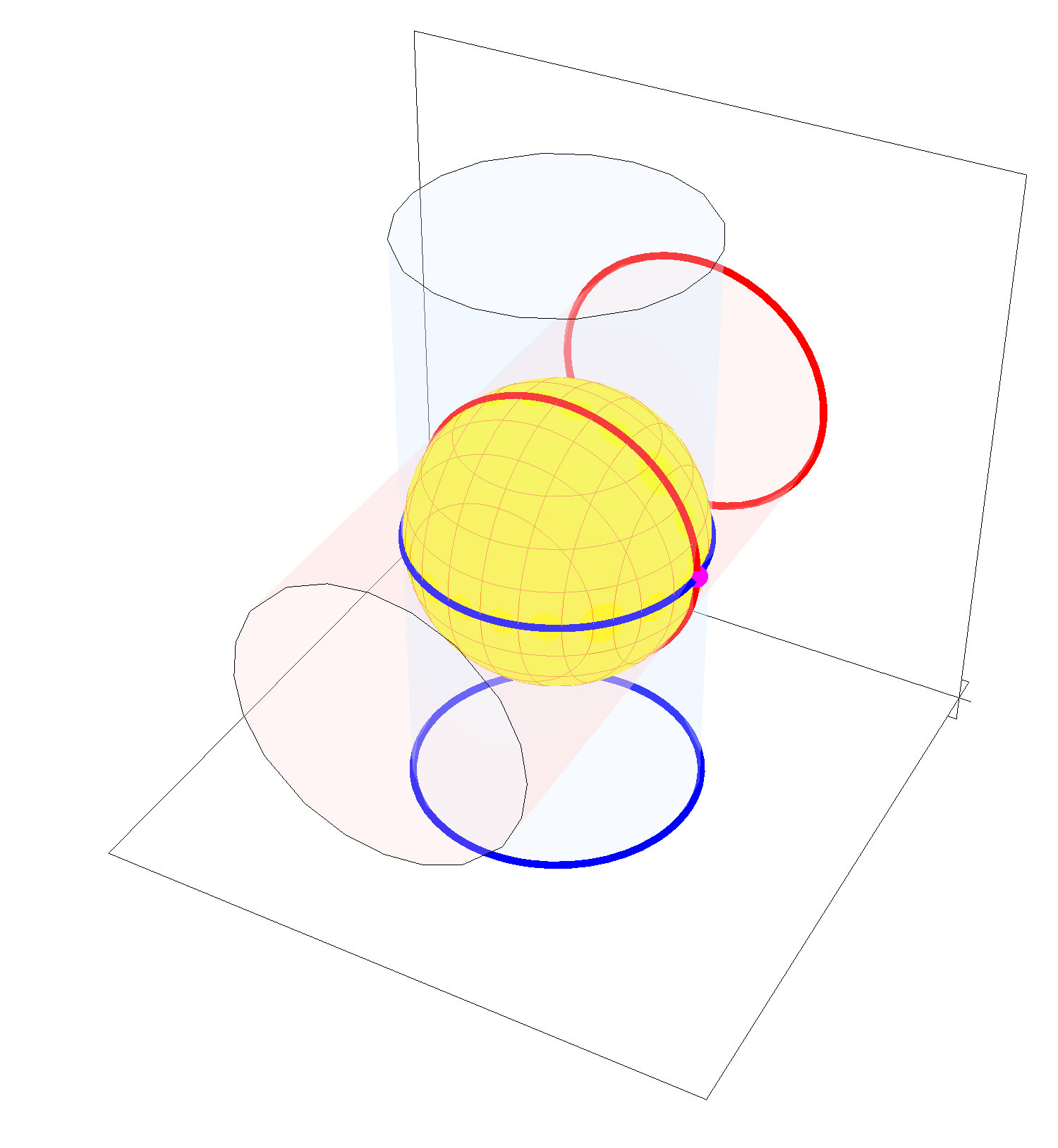}\\
  \includegraphics[width=0.34\textwidth]{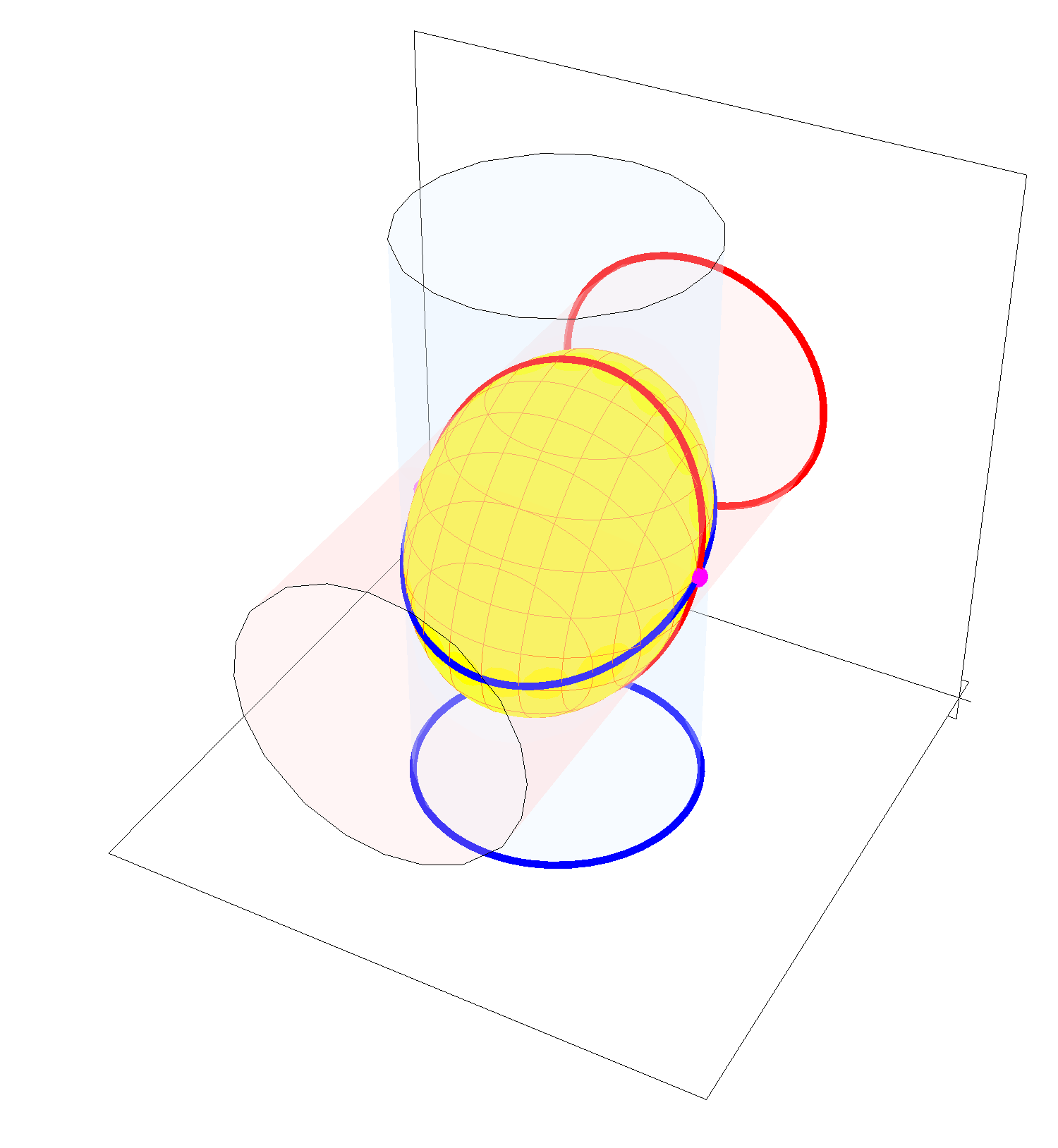}\hspace{3ex}
  \includegraphics[width=0.34\textwidth]{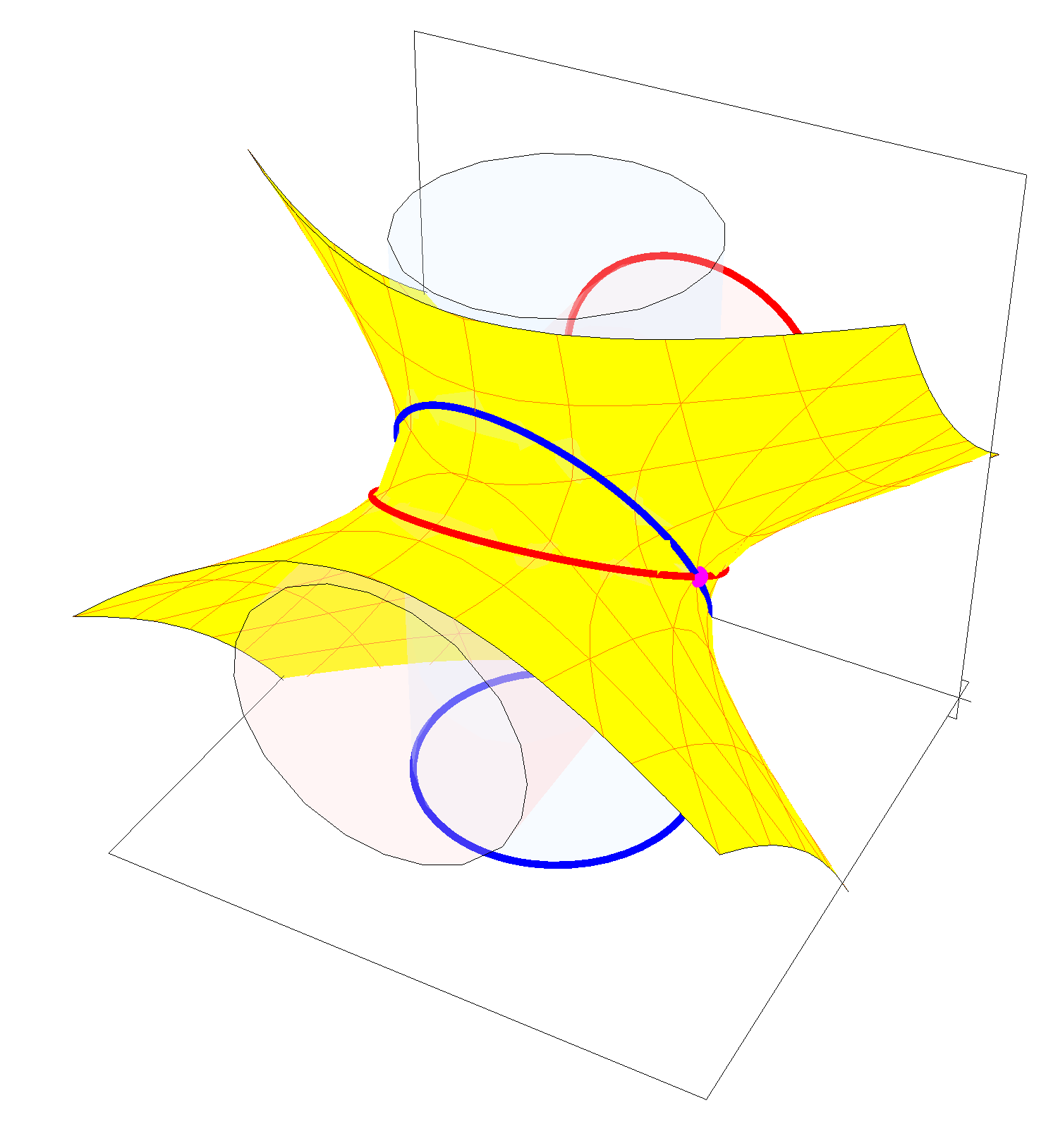}
\begin{minipage}{0.9\textwidth}
\caption{Reconstruction of quadric surface from two contour outlines. \label{fig outlines of quadric}}
\end{minipage}
\end{center}
\end{figure}

\begin{theorem}
  Three generic silhouettes determine a quadric uniquely,i.e., $s(2)=3$.  
\end{theorem}
\begin{proof}
 Let be given silhouettes $\av{S}_\f a$, $\av{S}_\f b$ and $\av{S}_\f c$. As above, it is possible to identify $\{\f p,\f q\}=\av{C}_\f a\cap\av{C}_\f b$ on the intersection of the corresponding cones. And similarly for the points $\{\f r,\f s\}=\av{C}_\f a\cap\av{C}_\f c$. Since silhouettes are considered to be generic we may assume that $\f c$ is not contained in the line joining $\f a$ and $\f b$ and it follows that $\{\f p,\f q\}\not=\{\f r,\f s\}$. Hence we know at least three distinct points on $\av{C}_\f a$. Moreover they lie on regular conic section and they span unique plane $\av{H}_\f a$. Finally we arrive at $\av{C}_\f a=\av{H}_\f a\cap(\f a\#\av{S}_\f a)$.
 
The process may be repeated to obtain e.g. $\av{C}_\f b$.  Now the statement follows from Theorem~\ref{thm quadric two contours}.
\end{proof}

\subsection{Contour curves on rational ruled surfaces}

Let $\av{R}$ be a rational ruled surface parameterized as in \eqref{eq R}. From \eqref{eq tangent planes} it is seen that the~tangent planes along a~regular ruling form a pencil. Hence for an arbitrary point not contained in the ruling there exists a unique tangent plane passing through this point. It turns out that the~contour curve is a~section and thus a~rational curve. In particular, the condition on the~point $\f a$ to be contained in the tangent plane of the~surface can be expressed as
\begin{equation}
   \det[\f a,\f p(s),\f q(s),t_0\dot{\f p}(s)+t_1\dot{\f q}(s)]=0.
\end{equation}
Solving this equation for $t_0$, $t_1$ and substituting back to \eqref{eq R} leads to the parameterization of the contour curve $\av{C}_\f a$ in the form
\begin{equation}\label{eq param kontury}
\f c_\f a(s)=\det[\f a,\f p(s),\f q(s), \dot{\f q}(s)]\f p(s)-\det[\f a,\f p(s),\f q(s), \dot{\f p}(s)]\f q(s).
\end{equation}
See \cite{PePoRa99} for the formula using Pl\"{u}cker coordinates. If $\f p(s)$ and $\f q(s)$ were minimal sections of $\av{R}$ of degrees $m$ and $m+n$ respectively, then \eqref{eq param kontury} leads to the following bound of the degree of the contour curves
\begin{equation}
  \deg\av{C}_\f a\leq 4m+2n-2=2\deg\av{R}-2,  
\end{equation}
where the equality holds for a generic contour curve on ruled surface without singular rulings. Quadrics are ruled surfaces with $m=1$ and $n=0$, and thus a generic contour shall be of degree $2$. As mentioned above this is true if the point $\f a$ does not lie on the quadric -- in which case $\av{C}_\f a$ consists of two lines. This behaviour is observed for ruled surfaces of higher degrees as well. If $\f a\in\av{R}$ or it is contained in the torsal tangent plane then $\av{C}_\f a$ is reducible -- it consists of the ruling and a section of the~degree one less.

Although contour lines are sections they do not posses a lowest possible degree and thus two contours always intersect.

\begin{lemma}\label{lem intersection on R}
 Two generic contours on a~rational ruled surface $\av{R}$ intersect at $\deg\av{R}$ regular points and in the cuspidal points of  torsal rulings.
\end{lemma}

\begin{proof}
  Let be given two generic contours $\av{C}_\f a$ and $\av{C}_\f b$. Recall that the Gauss image  $\gamma(\av{C}_\f a)$ of the contour is the section of dual surface $\av{R}^\vee$ by plane $\sum_{i=0}^3a_iy_i=0$ and similarly for the second contour $\av{C}_\f b$.  Hence $\gamma(\av{C}_\f a)\cap\gamma(\av{C}_\f b)$ is just the~intersection of $\av{R}^\vee$ with a~generic line. Since these are exactly the images of non singular intersections of $\av{C}_\f a$ and $\av{C}_\f b$ we conclude that two contours intersect in $\deg\av{R}^\vee=\deg\av{R}$ regular points.
  
The image of the~mapping $\gamma$ is not defined on singular locus of $\av{R}$, and thus the singular intersection cannot be seen from the dual surface.

  Let $\f p(s)$ and $\f q(s)$ be parametrizations of two minimal section and let the~torsal ruling correspond to the parameter value $s=s_0$. Write $\f p(s_0)=\f p_0$, $ \dot{\f p}(s_0)= \dot{\f p}_0$ and similarly for the $\f q(s)$. Since the ruling is torsal we have $\rank[\f p_0,\dot{\f p}_0,\f q_0,\dot{\f q}_0]=3$ and thus there exist constants $\lambda_1,\dots,\lambda_4$ such that $\lambda_1\f p_0+\lambda_2\dot{\f p}_0+\lambda_3\f q_0+\lambda_4\dot{\f q}_0$. After substituting to \eqref{eq param kontury} arrive at
\begin{equation}
\begin{array}{rcl}
  \f c_\f a(s_0) &=& \det[\f a,\f p_0,\f q_0, \dot{\f q}_0]\f p_0-\det[\f a,\f p_0,\f q_0,\frac{-1}{\lambda_2}
  (\lambda_1\f p_0+\lambda_3\f q_0+\lambda_4\dot{\f q}_0 )]\f q_0 \vspace{5pt} \\ 
  &=&\det[\f a,\f p_0,\f q_0, \dot{\f q}_0](\f p_0+\frac{\lambda_4}{\lambda_2}\f q_0).
\end{array}
\end{equation}
As $\f c_a(s_0)$ is a~point in the projective space we may write it as $\lambda_2\f p_0+\lambda_4\f q_0$. 

Consider an arbitrary section passing through this point, i.e. admitting a parameterization $\f r(s)=\varphi(s)\f p(s)+\psi(s)\f g(s)$ where $\varphi(s_0)=\lambda_2$ and $\psi(s_0)=\lambda_4$. Then the tangent line to this section at $s=s_0$ spanned by $\f r(s_0)$ and 
\begin{equation}
  \dot{\f r}(s_0)=\dot{\varphi}(s_0)\f p_0+\lambda_2\dot{\f p}_0+\dot{\psi}(s_0)\f q_0+\lambda_4\dot{\f q}_0=
  (\dot{\varphi}(s_0)-\lambda_1)\f p_0+(\dot{\psi}(s_0)-\lambda_3)\f q_0
\end{equation}
coincides with the ruling. Hence all the sections pass through $\f c_\f a(s_0)$ with the same direction and it must be the~unique cuspidal point of torsal ruling.
\end{proof}

Given a~proper parameterization $\f c(s)$ of a~rational curve $\av{C}$, the tangent planes to the ruled surface $t_0\f c(s)+t_1\f q(s)$ along $\av{C}$ contain the point $\f a$ if and only if $\f q(s)$ is contained in the plane spanned by a tangent line to $\av{C}$ at $\f c(s)$ and the~point $\f a$. It means that 
\begin{equation}
  \f q(s)= \phi(s)\f c(s)+\xi(s)\dot{\f c}(s)+\psi(s)\f a,
\end{equation}
where $\phi,\ \xi$ and $\psi$ are arbitrary polynomials. It turns out that there exists a plenty of ruled surfaces with prescribed contour. Surprisingly we still have:

\begin{theorem}\label{thm ruled determined by 2 contours}
  Two contour curves determine rational ruled surface uniquely, i.e., $c(k)=2$.
\end{theorem}

\begin{proof}
  Let $\f c_\f a(u)$ and $\f c_\f b(v)$ be proper parametrizations of two contour lines $\av{C}_\f a$ and $\av{C}_\f b$ of the unknown surface $\av{R}$. By Proposition~\ref{prp coherent parameterizations}, there exist reparametrizations such that $\f c_\f a(\phi(s))$ and $\f c_\f b(\psi(s))$ correspond in a parameter, i.e., $\f c_\f a(\phi(s))$ and $\f c_\f b(\psi(s))$ are points on the same ruling for each $s\in\C$. To find this reparametrization, realise that a tangent plane at a~point of $\av{C}_\f a$ is spanned by the~tangent line to the contour and point $\f a$. Moreover the tangent plane contains the ruling and thus the corresponding point on $\av{C}_\f b$ as well. Therefore it is contained in the intersection of $\av{C}_\f b$ with the tangent plane, which is expressed as
\begin{equation} \label{eq first coherent}
  \det[\f c_\f a(u),\dot{\f c}_\f a(u), \f a,\f c_\f b(v)]=0.
\end{equation}
This equation defines a curve in the space of parameters $u,\ v$. Thus a rational parameterization $u=\phi(s)$, $v=\psi(s)$ of any of its components leads to a parameterization of the~ruled surface $t_0\f c_\f a(\phi(s))+t_1\f c_\f b(\psi(s))$, such that $\av{C}_\f a$ is a contour w.r.t. $\f a$. The second curve is a contour w.r.t. $\f b$ if and only if $(\phi(s),\psi(s))$ is a parameterization of some component of the~curve
\begin{equation} \label{eq second coherent}
  \det[\f c_\f b(u),\dot{\f c}_\f b(u), \f b,\f c_\f a(v)]=0,
\end{equation}
by the same arguments. So, let $\Delta$ denotes the greatest common divisor of the left hand hand sides of \eqref{eq first coherent} and \eqref{eq second coherent} and let $\Delta=\Delta_1\cdots\Delta_k$ be a factorization to reducible polynomials. Then any $\Delta_i(u,v)$ defining a~rational curve in the space of parameters provides a parameterization of the~sought ruled surface. 

It remains to show that the solution is unique. By Proposition~\ref{prp coherent parameterizations} the parameterization of $\Delta_i$ can by written w.l.o.g. as $u=s$ $v=(\alpha s+\beta)/(\gamma s+\delta)$. Let $\av{C}_\f a$ and $\av{C}_\f b$ intersect in $k$ points $\{\vek[p]_i\}_{i=1}^k$ and let $\f p_i=\f c_\f a(s_i)$. As the points from the intersection lie on the same ruling, they must correspond in parameter, i.e, we have $k$ equations
\begin{equation}
  s_i=\psi(s_i)=\frac{\alpha s_i+\beta}{\gamma s_i+\delta},\qquad\text{for},\qquad i=1,\dots,k.
\end{equation}
Any linear rational function is determined by values at three points. By Lemma~\ref{lem intersection on R} $k\geq\deg\av{R}$  and thus $\psi(s)$ is unique whenever $\deg\av{R}\geq 3$.
\end{proof}

\begin{remark}
  Because the~rational component of the~curve $\Delta=0$ in the proof of Theorem~\ref{thm ruled     determined by 2 contours}, is given by a pair of linear rational functions we immediately conclude that it can be written as
\begin{equation}
 \alpha_{11} u v -\alpha_{10} u+\alpha_{01} v-\alpha_{00} = 0,
\end{equation}
for some constants $\alpha_{ij}$.
\end{remark}

Hence the number of needed contours  does not depend on the degree and we have $c(k)=2$ for all $k$. On the other hand the~reconstruction from silhouettes differs from the quadric case in an~unexpected way.

\begin{theorem}
  Ruled surface of degree at least $3$  is uniquely determined  by two its generic silhouettes.
\end{theorem}

\begin{proof}
  Let $\f s_\f a(s)$ be a proper parameterization of silhouette $\av{S}_\f a$. This is a projection of the~contour $\av{C}_\f a$ from a point $\f a$ to a fixed plane and thus the contour can be lifted
  \begin{equation}\label{eq lift contour}
    \f c_\f a(s)=\phi(s)\f s_\f a(s)+\psi(s)\f a,
  \end{equation}
  where $\phi$ and $\psi$ are  polynomials. Parameterizations of two silhouettes $\f s_\f a(s)$, $\f s_\f b(s)$ correspond in parameter if they are projections of parameterizations of contours that correspond in parameter. First, we observe that two such  parameterizations enable to reconstruct the contours as 
\begin{equation}\label{eq contour from 2 silhouettes}
    \f c_\f a(s)=\det[\f b,\f s_\f a(s),\f s_\f a'(s),\f a]\f s_\f a(s)+\det[\f b,\f s_\f a(s),\f s_\f a'(s),\f s_\f b(s)]\f a.
\end{equation}
and analogously for $\f c_\f b(s)$.
To prove this consider the tangent plane to the sought ruled surface at point $\f c_\f b(s)$. We know that it is  spanned by tangent line to $\av{C}_\f b$ at this point and the point $\f b$. Because of the~linearity of the projection, the tangent line to $\av{S}_\f b$ at $\f s_\f b(s)$ together with the point $\f a$ span the same plane. Hence it possesses an~equation
\begin{equation}\label{eq tangent from silh}
  \det[\f x,\f b,\f s_\f b(s), \dot{\f s}_b(s)]=0.
\end{equation}
If a parameterization of $\av{C}_\f a$ corresponds in parameter with $\f c_\f b(s)$, then $\f c_\f a(s)$ must lie in tangent plane \eqref{eq tangent from silh} for all $s$. Substituting \eqref{eq lift contour} into \eqref{eq tangent from silh} we arrive at conditions on polynomials $\phi$ and $\psi$
\begin{equation}\label{eq phi psi}
  \phi(s)\det[\f s_\f a(s),\f b,\f s_\f b(s), \dot{\f s}_b(s)]+\psi(s)\det[\f a,\f b,\f s_\f b(s), \dot{\f s}_b(s)]\equiv 0.
\end{equation}
One obvious solution of the~above equation is exactly \eqref{eq contour from 2 silhouettes}. Note that pairs $(\phi,\psi)$ fulfilling \eqref{eq phi psi} form an~ideal in $\C[s]^2$. Since $\C[s]^2$ is a~principal ideal domain, each  solutions would  be the same up to multiplication by some polynomial. The effect on parametrization \eqref{eq lift contour} consists in adding some base points. However the curve parametrized by \eqref{eq contour from 2 silhouettes} is unique.

Let be given two proper parameterizations $\f s_\f a(s)$ and $\f s_\f b(u)$ of the silhouettes.
The reconstruction of the surface is based on the existence of reparameterization $\phi(s)$ such that $\f s_\f a(s)$ and $\f s_\f b(\phi(s))$ correspond in parameter. And thus the uniqueness of the sought surface follows from the uniqueness of the reparameterization $\phi$. First observe that $\phi$ is a~linear fractional transformation by Proposition~\ref{prp coherent parameterizations}.  Such a $\phi$ is then uniquely determined by values at three points and it is enough to find three corresponding pairs $(\f p_i,\f q_i)\in\av{S}_\f a\times\av{S}_\f b$.

Let $\f r\in\av{C}_\f a\cap\av{C}_\f b$ be a regular point in the intersection of two contours. The point is regular on both contours and as well are its projections on the silhouettes. Moreover the tangent plane $T_\f r\av{R}$ is spanned by $\f r$, $\f a$ and $\f b$.  Conversely any tangent plane passing through $\f a$ and $\f b$ contains  such a point $\f r$. Recall that tangent planes along $\av{C}_\f b$ are parameterized by \eqref{eq tangent from silh} and the analogous equation can be written for tangents along $\av{C}_\f a$. Solution of
\begin{equation}\label{eq marks on sil}
  \det[\f a,\f b,\f s_\f b(s), \dot{\f s}_b(s)]=0.
\end{equation} 
consists of singular points on $\av{S}_\f b$ and regular point where the tangent contains $\f a$. As the regular points correspond to the regular intersections of contours, there exists $\ell = \deg\av{R}$ such points by Lemma~\ref{lem intersection on R}.  

Assume $P=\{\f p_i\}_{i=1}^\ell\subset\av{S}_\f a$ and $Q=\{\f q_j\}_{j=1}^\ell\subset\av{S}_\f b$ be sets of these points. 
Because the~silhouettes $\av{S}_\f a$ and $\av{S}_\f b$ are considered to be generic we may assume that point $\f b$ is not contained in the~plane spanned by $\f a$, $\f p_i$ and $\f q_j$ for each pair $i,j$ of indices. And similarly for the~point $\f a$. This ensures that for each $i$ there exists a unique $j$ such that lines $\f a\f p_i$ and $\f b\f q_j$ do intersect. This provides a corresponding pair of points $(\f p_i,\f q_j)$. As we have $\ell=\deg\av{R}$ such pairs the sought reparameterization is unique whenever $\deg\av{R}\geq 3$ which completes the proof.

\end{proof}

\begin{example}
 Let $\f a=(1:2:1:2)$ and the first silhouette be parameterized
 \begin{equation}
   \f s_\f a(s)=(-8+ 16s- 10  s^2 + 2 s1^3: 0: 
 8  s^2 - 6 s^3 + s^4: -4 + 24  s - 
  18s^2 + 4 s^3).
 \end{equation}
 The second silhouette w.r.t. $\f b=(1:2:0:1)$ admits parameterization
 \begin{equation}
  (-2: -4 - u + 2 u^2 - u^3: 1 + u - u^2 - u^3: 0).
 \end{equation}
 Substituting into \eqref{eq marks on sil} we obtain the candidates for the corresponding points. After removing singularities of silhouettes we arrive at
 \begin{equation}
   P=\{(1: 0: 2 +\I:3-\I),\ (1: 0: 2-\I: 3 +\I),\ (0: 0: 0:1)\}
 \end{equation}
  and
  \begin{equation}
   Q=\{(1: 3: -1+\I: 0),\ (1:3:-1-\I: 0),\ (1: 2: 0: 0)\}
  \end{equation}
  The efficient way to check the mutual positions of lines $\f a\f p_i$ and $\f b\f q_j$ for $i,j=1,\dots,3$ is to introduce Pl\"{u}cker coordinates, identifying lines in $\pr{3}$ with a point on the~quadric $\mathrm{Gr}(1,3)\subset\pr{5}$ -- see e.g. \cite[Chapter 2]{PoWa01} for the introduction to the topic. There is a bilinear form product $\langle\cdot,\cdot\rangle:\mathrm{Gr}(1,3)\times\mathrm{Gr}(1,3)\rightarrow\C$ such that $\langle\f X,\f Y\rangle=0$ if and only if the lines with Pl\"{u}cker coordinates $\f X$ and $\f Y$ intersect. Denote $\f P_i$ the coordinates of lines joining $\f a$ and $\f p_i$ and similarly $\f Q_j$ for lines $\f b\f q_j$. Then we have a matrix
  \begin{equation}
    \left( \langle\f P_i,\f Q_j\rangle\right)_{i,j=1}^3=
    \left(
      \begin{array}{ccc}
         -\I& 0 &1\\
         0& \I &1\\
         -1&-1&0
      \end{array}
    \right)
  \end{equation}
From this it is immediately seen that the corresponding pairs of points are $\{(\f p_1,\f q_2),(\f p_2,\f q_1),(\f p_3,\f q_3)\}$. Since we have $(\f p_1,\f q_2)=(\f s_a(1+\I),\f s_b(\I))$, $(\f p_2,\f q_1)=(\f s_a(1-\I),\f s_b(-\I))$ and $(\f p_3,\f q_3)=(\f s_a(2),\f s_b(1))$ the sought reparameterization is $u(s)=s-1$. The corresponding parameterizations of silhouettes can by lifted to contours by \eqref{eq contour from 2 silhouettes} which yields
  \begin{equation}
    \f c_a(s)=\left(-4 (2 - 3 s + s^2): -4 s (2 - 3 s + s^2): 
 s (-4 + 14 s - 8 s^2 + s^3): -4 + 16 s - 6 s^2\right)
  \end{equation}
  and
  \begin{equation}
    \f c_b(s)=\left(-2 + 3 s - s^2: -s (2 - 3 s + s^2): -(-2 + s) s^2: -(-3 + s) s\right)
  \end{equation}
  These parameterization of contour curves already correspond in the parameter and thus can be used to parameterize the~surface $\av{R}$ as $t_0\f c_\f a(s)+t_1\f c_\f b(s)$. Note however that this parameterization is not the optimal in the sense of the lowest possible degree.  In particular, the minimal sections on this cubic ruled surface  provide the parameterization $ (2 t_0 + t_1: 2 s t_0 + s t_1: s^2 t_0 + s t_1: t_1)$.
\end{example}

\section{Isophotes}
Recall that the isophote was defined as a loci of points whose normal direction encloses angles $\phi$ or $-\phi$ with the fixed direction. We used this modified definition to ensure that it is an algebraic curve. However it can happen in some cases, that
$\av{I}_{\f a,\alpha}$ decomposes into two components corresponding to the~choice of the sign of the angle. If this happen, both curves are algebraic and the usual definition can be used. However we will relate this behaviour with the reducibility of the offset of a surface and conclude that it is vary rare.

Let $\delta$ be a constant, \emph{$\delta$--offset} $\av{O}_\delta\av{X}_a$ of surface $\av{X}_a\subset\af{3}=\pr{3}\backslash\omega$ is defined as the closure of the set
\begin{equation}
  \{ P\pm\delta\frac{\nu(P)}{\Vert\nu(P)\Vert}\mid P\in(\av{X}_a)_{sm}\},
\end{equation} 
where $\nu(P)$ is a normal vector at point $P$ and $\Vert\nu\Vert=\sqrt{\nu_1^2+\nu_2^2+\nu_3^2}$.  By the offset of a~projective surface $\av{X}$ we mean the projective closure of $\av{O}_\delta\av{X}_a$. It is known that the offset of a~surface may consist of at most two components, however for a~generic surface it is an~irreducible variety. See e.g. \cite{ArSeSe97,Po95,VrLa14a} for more details about offsets of surfaces.

\begin{lemma}\label{lem offset then isophote}
  Let a surface $\av{X}$ has the~reducible offsets. Then the isophotes on the surface fall into two components corresponding to angles $\pm\phi$.
\end{lemma}

\begin{proof}
  Clearly it is  sufficient to prove the statement for the affine surface $\av{X}_a$. Let it be given by an~equation $f(x_1,x_2,x_3)=0$.  Then we may write $\nu(\f P)=\nabla\,f(\f P)$. It was proved in \cite{VrLa13} that the offset is reducible if and only if $\nabla\,f\cdot\nabla\,f$ is a perfect square in $\C[\av{X}_a]$, i.e. we may write $\nabla\,f\cdot\nabla\,f\equiv \sigma^2 \mod f$ for some polynomial $\sigma$. (The mentioned result is proved for the case of plane curves only, but it can be directly generalized for any hypersurface). Hence the isophote is the closure of the intersection $\av{X}_{sm}$ with the~reducible surface 
\begin{equation}
  \left(\nabla f(\f x)\cdot\f a-\alpha\sqrt{\f a\cdot\f a}\,\sigma(\f x) \right)\left(\nabla f(\f x)\cdot\f a+\alpha\sqrt{\f a\cdot\f a}\,\sigma(\f x)\right)=0,
\end{equation}
whose two components correspond to the sought components of the~isophote.
\end{proof}

Although there is a strong belief that the converse of Lemma~\ref{lem offset then isophote} is true as well we will prove it for the case of rational surfaces only.

\begin{lemma}\label{lem isophote then offset}
Let $\av{X}$ be a rational surface whose isophotes generically splits into components corresponding to the choice of the sign of angle. Then the offsets of the surface are reducible.
\end{lemma}

\begin{proof}
  Let $X(s,t)$ be an arbitrary proper parameterization of the affine surface $\av{X}_a$. The rational surface has the~reducible offset if and only if some (and thus any) proper parameterization fulfils the so called Pythagorean Normal property
   \begin{equation}
     \f n(s,t)^2:=\left(\partial_s X(s,t)\times\partial_t X(s,t)\right)^2=\sigma^2(s,t)
   \end{equation}
   for some rational function $\sigma$, see~\cite{ArSeSe97} for the proof. If we set $\f n=(n_1/n_0,n_2/n_0,n_3/n_0)$ and $f=\gcd(n_1,n_2,n_3)$ we may write $\f n = f \hat{\f n}/n_0$. The pull--back of the isophote $\av{I}_{\f a,\alpha}$ to the affine plane is then a curve given by the equation
    \begin{equation}\label{eq posledni}
       (\hat{\f n}(s,t)\cdot\f a)^2-\f a^2\hat{\f n}(s,t)^2=0.
    \end{equation}
    Because in $\af{2}$ is every curve given by a single equation and for each direction $\f a$ and angle $\alpha$ the equation \eqref{eq posledni} can be decomposed into two polynomial factors 
   \begin{equation}
    \left(\hat{\f n}(s,t)\cdot\f a-\alpha\sqrt{\f a^2}\sqrt{\hat{\f n}(s,t)^2}\right)\left(\hat{\f n}(s,t)\cdot\f a+\alpha\sqrt{\f a^2}\sqrt{\hat{\f n}(s,t)^2}\right)=0
   \end{equation}
   we conclude that $\hat{\f n}(s,t)^2=\sigma^2(s,t)$ for some polynomial $\sigma$. Hence the Pythagorean  Normal property $\f n^2=(f\sigma/n_0)^2$ is fulfilled and the offset of $\av{X}$ consists of two components.

\end{proof}

\begin{corollary}\label{cor sphere}
  The only real rational non-developable ruled surface $\av{R}$ with reducible isophotes is a~sphere.
\end{corollary}

\begin{proof}
 The only quadric with reducible offset is known to be a~sphere. Assume now, that the $\deg\av{R}>2$. Since it is real rational ruled surface  there exists a real proper parameterization of an~affine part
  \begin{equation}
    X(s,t)=P(s)+t  \overrightarrow{q}(s).
  \end{equation}
it is enough to show that the condition $\f n(s,t)^2=\sigma^2(s,t)$ is fulfilled for the sphere only.
Let us denote $\f n(s,t)=\f n_1(s)+t\f n_2(s)$, where $\f n_1=P'\times\overrightarrow{q}$ and $\f n_2=\overrightarrow{q}'\times\overrightarrow{q}$. Then $\f n^2(s,t)$ is a quadratic polynomial in $t$  and it is a~perfect square only if its discriminant w.r.t parameter $t$ vanishes identically
   \begin{equation}
     (\f n_1\cdot\f n_2)^2-\f n_1^2\, \f n_2^2\equiv 0.
   \end{equation}
   As $\f X(s,t)$ is a real parameterization the~equation is fulfilled if and only if $\f n_1(s,t)=\lambda(s,t)\f n_2(s,t)$ or at least one of $\f n_i^2(s,t)$ vanishes.
   If  $\f n_1$ and $\f n_2$ are linearly dependent then the tangent planes are constant along the~rulings and thus the surface is developable. Similarly $\f n_i^2\equiv 0$ implies $\f n_i\equiv(0,0,0)$, because the parameterization is real. And it is easily seen that the surface must be again developable.  Hence we conclude that there does not exist ruled surface  of degree larger than two with reducible offset and the corollary is proved. 
\end{proof}

\begin{remark}
  The assumption on the surface to be real is essential in Corrolary~\ref{cor sphere}. For example the surface $x_0^2 x_2 + x_1^2 x_2 + 2 x_0 x_2^2 + x_2^3 + \I x_0^2 x_3 - \I x_1^2 x_3 -\I x_2^2 x_3 + 2 x_0 x_3^2 + x_2 x_3^2 - \I x_3^3=0$  is a cubic ruled surface with parametrization $t_0(1:s:0:0)+t_1(0:2s:1-s^2:\I(1+s^2))$. Direct computation verifies that it is A~ non-developable surface with reducible offset, and thus by Lemma~\ref{lem offset then isophote} its isophotes split into two components.
\end{remark}

The isophotes on the~sphere consist of two circles -- they are intersections of the~sphere with a~circular cone with vertex located at the~center of the~sphere. Hence both components are rational.  Some other examples of surfaces with rational isophotes were studied in \cite{Aigner2009}. Unfortunately isophotes are typically curves of higher genus, as the following theorem proves for rational ruled surfaces.  

\begin{theorem}\label{thm genus of isophote}
Let $\av{R}$ be a real rational non--developable ruled surface different from the~sphere. And let the plane at infinity $\omega$ contains $k$ rulings (counting the multiplicities) of the surface then the generic isophote on $\av{R}$ is a~hyperelliptic curve of genus
\begin{equation}
  \genus(\av{I}_{\f a,\alpha})=\deg \av{R}-k-1.
\end{equation}
\end{theorem}

\begin{proof}
  Denote by $\widehat{\av{I}}_{\f a,\alpha}$ the conic section in $\pr{2}$ given by equation
  \begin{equation}
     (\sum_{i=1}^3 a_ix_i)^2-\alpha^2(\sum_{i=1}^3a_i^2)(\sum_{i=1}^3x_i^2)=0,
  \end{equation}
  i.e., it is a closure of the image of the~isophote $\av{I}_{\f a,\alpha}$ under normal mapping~\eqref{eq normal mapping}. Now, observe that that the normal mapping sends a generic ruling $\av{L\subset R}$ to the line.  As the~tangent planes along $\av{L}$ form a~pencil, i.e. $\gamma(\av{L})$ is a line in $(\pr{3})^\vee$ and $\nu$ is just a composition of the~Gauss map with the projection, we conclude that $\nu(\av{L})\subset\pr{2}$ is indeed a line.  It intersects the conic $\widehat{\av{I}}_{\f a,\alpha}$ in two points and thus the isophote $\av{I}_{\f a,\alpha}=\nu^{-1}(\widehat{\av{I}}_{\f a,\alpha})$ intersects rulings in two points.
  
   Because $\av{R}$ is a~rational ruled surface, there exists a mapping $\phi:\av{R}\dashrightarrow\pr{1}$ whose fibers are exactly the rulings. (to see this take e.g. inverse of the proper parameterization $t_0\f p(s)+t_1\f q(s)$.) The restriction $\phi\mid_{\av{I}_{\f a,\alpha}}$ is then~a double cover of $\pr{1}$ and thus the isophote is a~hyperelliptic curve.
  
Using Riemann--Hurwitz formula it is possible to express the genus of the~isophote as
\begin{equation}\label{eq RH}
  \genus(\av{I}_{\f a,\alpha})=\frac{1}{2} \#\{\text{ramification points of }\phi\mid_{\av{I}_{\f a,\alpha}}\}-1.
\end{equation}
The mapping is ramified over the points, where the~isophote intersects the ruling at one point with multiplicity two. In other words the number of ramification points is the number of rulings $\av{L}$ such that $\nu(\av{L})$ is tangent to $\widehat{\av{I}}_{\f a,\alpha}$. In order to calculate the number of lines in the family tangent to the conic we pass to the dual space. 
Looking carefully at the definition of mappings $\gamma$ and $\nu$ we immediately see that the dual of the~line $\nu(\av{L})$ is a point $(p_1:p_2:p_3)$, for which $\av{L}\cap\omega = (0:p_1:p_2:p_3)$.  Therefore the family of rulings can be identified with a section $\av{D}$ of the surface $\av{R}$ by the plane at infinity -- it is $\av{R}\cap\omega$ with rulings contained in $\omega$ removed.  Hence $\deg\av{D}=\deg \av{R}-k$. Since $\widehat{\av{I}}_{\f a,\alpha}$ is a~regular conic for a generic choice of $\f a$ and $\alpha$ its dual is again a regular conic section and $\deg\widehat{\av{I}}_{\f a,\alpha}^\vee=2$.
The number of ramification points is then by B\'{e}zout theorem
\begin{equation}
  \#\left(\widehat{\av{I}}_{\f a,\alpha}^\vee\cap\av{D}\right)=(\deg\widehat{\av{I}}_{\f a,\alpha}^\vee)\cdot(\deg\av{D})=2(\deg\av{R}-k).
\end{equation} 
Substituting this into \eqref{eq RH} proves the theorem.
\end{proof}

\begin{corollary}
  A generic isophote on a~real rational ruled surface $\av{R}$ is a rational curve if and only if $\av{R}$ is a sphere or the coordinates of $\overrightarrow{q}(s)$ in the parameterization $P(s)+t \overrightarrow{q}(s)$ of affine part are linear polynomials.
\end{corollary}

\begin{proof}
 Assume $\av{R}$ not being the sphere. Then by Theorem~\ref{thm genus of isophote} the generic isophote is rational if and only if there is $\deg\av{R}-1$ rulings contained in $\omega$. Because $\av{R}\cap\omega$ is a curve of degree $\deg\av{R}$ the~section of the surface by plane at infinity must be a line and $\overrightarrow{q}(s)$  is its parametrization.
\end{proof}

\begin{corollary}
  The number of components of the~real part of the~isophote on $\av{R}$ after desingularization is at most $\deg\av{R}-k$
\end{corollary}

\begin{proof}
  For a hyperelliptic curve $\av{C}$ defined over $\R$ of genus $g$ there exists a real~birational map $\av{E}\rightarrow\av{C}$, where $\av{E}$ is the~plane  curve in a Weierstrass form $y^2=p(x)$ for some square--free polynomial $p$ of degree $2g+2$. 
Since the number of real components is birational invariant and $\av{E}$ has at most $\frac{1}{2}\deg p$ real connected components, the corollary is proved.
\end{proof}

\begin{example}\label{ex quadric rational iso}
  Let $\av{Q}$ be an elliptic or hyperbolic paraboloid, i.e., a quadric given by equation $b^2x_1^2\pm a^2x_2^2-a^2b^2x_0x_3$.
  Its intersection with the~plane at infinity consists of two lines -- one is counted as the section and the second as the ruling. Thus by Theorem~\ref{thm genus of isophote}  the genus of generic isophote  equals to $\deg\av{Q}-1-1=0$. In fact the normal mapping 
  \begin{equation}
    \nu:(x_0:x_1:x_2:x_3)\mapsto(2b^2x_1^2:\pm 2a^2 x_2^2:-a^2b^2x_0)
  \end{equation}
is birational in this particular case. This confirms that the isophote is a~rational curve  as it is the~pull--back of a~conic section $\widehat{\av{I}}_{\f a,\alpha}$.
\end{example}

\begin{example}\label{ex quadric elliptic iso}
  For the remaining regular quadrics the intersection $\av{Q}\cap\omega$ is a~regular conic section and thus no ruling of $\av{Q}$ is contained in the infinity. Hence $\genus(\av{I}_{\f a,\alpha})=\deg\av{Q}-1=1$ and a generic isophote is an~elliptic curve. with at most two real connected components.
\end{example}

Note that whereas a~generic isophote on a~quadric $\av{Q}$ from Example~\ref{ex quadric elliptic iso} is elliptic, it does not mean that $\av{Q}$ does not possesses any rational isophote at all. Let $\av{D}$ be a curve from proof of Theorem~\ref{thm genus of isophote}, i.e., it can be identified with the regular conic section $\av{Q}\cap\omega$. Since $\widehat{\av{I}}_{\f a,\alpha}^\vee$ intersects $\av{D}$ in four points the isophote were considered as a double cover of $\pr{1}$ ramified at four points. However if $\widehat{\av{I}}_{\f a,\alpha}^\vee$ is tangent to $\av{D}$ there remains only two ramificacion points and the corresponding isophote is rational. (Its Weierstrass form would be $y^2-x^2p(x)$, where $p(x)$ is square-free of degree two. Hence a birational transform $\hat{y}=y/x$, $\hat{x}=x$ maps it to the~conic $\hat{y}^2=p(\hat{x})$.)

The set of conic sections may be identified with $\pr{5}$. Denote $\Phi$ the set of all conics tangent to the conic $\av{D}$. It is well known that $\Phi$ is a hypersurface of degree $6$. Next, by $\Psi$ we denote the closure of all the points in $\pr{5}$ corresponding to conic sections $\widehat{\av{I}}_{\f a,\alpha}^\vee$. Hence it is a set of conics dual to
\begin{equation}\label{eq pencil vole}
  \lambda(a_1x_1+a_2x_2+a_3x_3)^2+\mu(x_1^2+x_2^2+x_3^2+x_4^2).
\end{equation}
Simple computation shows that for each $(a_1:a_2:a_3)$ the pencil \eqref{eq pencil vole} is self dual and thus \eqref{eq pencil vole} parameterizes in fact the variety $\Psi$. Set $\av{V}\subset\pr{5}$ to be the set of double planes $(\sum a_ix_i)^2=0$. It is  Veronese surface and $\Psi$ is a cone over base $\av{V}$, whose  vertex is $\sum x_i^2$.  Hence $\Psi$ is a subvariety of $\pr{5}$ of dimension $3$ and degree $4$. The properties intersection product on subvarieties of $\pr{5}$ tell us that $\Phi\cap\Psi$ should be a surface of degree $6\cdot 4=24$. Realize that every double plane is tangent to every conic and thus $\av{V}\subset\Phi$, too. Luckily enough $\dim\av{V}=2$ and thus we have $\Phi\cap\Psi=\av{V}\cup\Xi$, where $\Xi$ is a surface of degree $20$ parameterizing rational isophotes on the quadric $\av{Q}$.

\section{Concluding remarks}

This paper was devoted to the study of contours and isophotes on ruled surfaces. We also presented the existence of the solution to the~reconstruction problems. Although ruled surfaces are used a lot in applications, contours and isophotes on other commonly used surfaces are not rational. There might be two possible directions of further research. First one is the~study of surfaces containing a lot of rational contour lines -- for example on dual to Del Pezzo surface there exists at least two parametric family of rational contour curves. Second, we can leave the strong assumption on rationality and focus on surface with manageable contour curves -- for example hyperelliptic curves posses relatively simple non-rational parameterizations, which allows to extend the class of studied surfaces e.g. by envelopes of quadratic cones and mentioned duals to Del Pezzo surfaces. The method for efficient rational approximation of hyperelliptic contour curves could be used to approximate isophotes on ruled surfaces as well.

\section*{Acknowledgements}

The author is supported by the project LO1506 of the Czech Ministry of Education, Youth and Sports.

%
%

\bibliographystyle{siam}
\bibliography{bibliography,bibliographyML}

\end{document}